\documentclass[11pt]{amsproc}
\usepackage{amsmath}
\usepackage{amsthm}
\usepackage{amssymb}
\usepackage[abbrev]{amsrefs}
\usepackage{mathtools}
\usepackage{mathrsfs}
\usepackage[most]{tcolorbox}
\usepackage{fullpage}
\usepackage{enumerate}
\usepackage[hypertexnames=false]{hyperref}
\usepackage{thmtools}
\usepackage{url}
\usepackage{float}
\usepackage{physics}

\usepackage[capitalize,nameinlink,noabbrev,nosort]{cleveref}
\hypersetup{
	colorlinks=true,       % false: boxed links; true: colored links
	linkcolor=red!40!black,         % color of internal links
	citecolor=red!40!black,        % color of links to bibliography
	filecolor=red!40!black,      % color of file links
	urlcolor=red!40!black,           % color of external links
}
\renewcommand{\PrintReviews}[1]{}

\newtheorem{theoremcounter}{Theorem Counter}[section]

\theoremstyle{remark}
\newtheorem{remark}[theoremcounter]{Remark}

\theoremstyle{definition}
\newtheorem{definition}[theoremcounter]{Definition}
\newtheorem{example}[theoremcounter]{Example}

\theoremstyle{plain}
\newtheorem{lemma}[theoremcounter]{Lemma}
\newtheorem{proposition}[theoremcounter]{Proposition}
\newtheorem{corollary}[theoremcounter]{Corollary}

\newtheorem{theorem}[theoremcounter]{Theorem}

\crefname{equation}{eq.}{eqs.}
\Crefname{equation}{Eq.}{Eqs.}
\numberwithin{equation}{section}

\DeclareMathOperator{\ser}{\vartheta}
\DeclareMathOperator{\ord}{\mathrm{ord}}
\DeclareMathOperator{\sgn}{\mathrm{sgn}}

\newcommand{\Gz}[1]{\Gamma_0(#1)}
\newcommand{\Fricke}[1]{\Gamma_0^*(#1)}

\newcommand{\wM}[2]{\mathcal{M}_{#1}^{!}(#2)}
\newcommand{\Z}[0]{\mathbb{Z}}
\newcommand{\SLZ}[0]{\mathrm{SL}_2(\Z)}
\newcommand{\F}[0]{\mathcal{F}}
\newcommand{\uH}[0]{\mathbb{H}}
\newcommand{\A}[0]{\mathcal{A}}
\newcommand{\Q}[0]{\mathbb{Q}}
\newcommand{\C}[0]{\mathbb{C}}

\begin{document}

%%%%%%%%%%%%%%%%%%%%%%%%%%
\title{The Serre Derivatives and Zeros of Modular Forms}
\author{Naoki Sugibayashi}
\address{Graduate School of Mathematics, Kyushu University, Motooka 744, Nishi-ku, Fukuoka 819-0395, Japan}
\email{sugibayashi.naoki.238@s.kyushu-u.ac.jp}

%%%%%%%%%%%%%%%%%%%%%%%%%%

%%%%%%%%%%%%%%%%%%%%%%%%%%
\maketitle
%%%%%%%%%%%%%%%%%%%%%%%%%%

%%%%%%%%%%%%%%%%%%%%%%%%%%
\begin{abstract}
Since the work of F. Rankin and Swinnerton-Dyer on the zeros of Eisenstein series, many results have been obtained concerning the zeros of modular forms.
In this paper, we study the zeros of Serre derivatives of modular forms.
In particular, we prove that if all the zeros of a weakly holomorphic modular form in the standard fundamental domain lie on the lower boundary, then the same property holds for its Serre derivative.
\end{abstract}
%%%%%%%%%%%%%%%%%%%%%%%%%%

\section{Introduction}
Let $\uH$ denote the upper half-plane. For an even integer $k\geq4$, the normalized Eisenstein series of weight $k$ for $\SLZ$ is a holomorphic function on $\uH$ defined by
\[
E_k(\tau):=1-\frac{2k}{B_k}\sum_{n=1}^{\infty}\left(\sum_{0<d|n}d^{k-1}\right)q^n\qquad(\tau\in\uH,\quad q:=e^{2\pi i\tau}),
\]
where $B_k$ is the $k$-th Bernoulli number. The series $E_k$ is a modular form of weight $k$ for $\SLZ$.
The study of the zeros of Eisenstein series $E_k$ began with the pioneering work of Wohlfahrt~\cite{Woh63} and R. Rankin~\cite{Ran69}.
A major breakthrough was then achieved by F.~Rankin and Swinnerton-Dyer~\cite{RSD70}. They showed that for any even integer $k\geq4$, all the zeros of the Eisenstein series $E_k$ in the standard fundamental domain lie on the lower boundary arc.

The method developed in \cite{RSD70} has since become a foundation for subsequent work on zeros of modular forms.
For example, R. Rankin~\cite{Ran82} generalized this method to certain Poincaré series.
Moreover, Asai, Kaneko, and Ninomiya \cite{AKN97} studied the zeros of a certain natural basis of weakly holomorphic modular forms of weight $0$.
In 1998, Kaneko and Zagier \cite{KZ98} studied the Atkin orthogonal polynomials and ``hypergeometric modular forms", in terms of their relation to the $j$-invariants of supersingular elliptic curves as well as the distribution of their zeros.
Furthermore, many other results are known concerning the zeros of modular forms for $\SLZ$~\cites{Get04,Rud05,Gun06,DJ08,RVY17,Kla19,XZ21,Han23,XZ25,Rav25}.
In addition, there has been considerable research in which the group $\SLZ$ is replaced by Fricke groups~\cites{MNS07,Shi07,BKM08,Shi10,CI16,HK21,Kug21,Kug22}.
In another direction, Aricheta, Celeste, and Nemenzo~\cite{ACN16} studied the behavior of the zeros of modular forms under Hecke operators.

On the other hand, it is known that the ordinary differential operator $D:=\frac{1}{2\pi i}\frac{d}{d\tau}$ does not preserve modularity.
There are, however, two standard ways to compensate for this failure.
The first is to add a non-holomorphic correction term involving $\Im(\tau)$.
The second is to use the quasi-modular Eisenstein series of weight $2$
\[
E_2(\tau):=1-24\sum_{n=1}^{\infty}\left(\sum_{0<d|n}d\right)q^n,
\]
which leads to the (Ramanujan--)Serre derivative of a weight $k$ modular form $f$
\[
\ser_k(f):=D(f)-\frac{k}{12}E_2f.
\]

In this paper, we study the zeros of the Serre derivative $\ser_k(f)$.
In particular, we prove that if all the zeros of a weakly holomorphic modular form $f$ in the standard fundamental domain lie on the lower boundary, then the same is true for the Serre derivative $\ser_k(f)$. 
To the best of our knowledge, this is the first result on the behavior of the zeros of modular forms under a differential operator that preserves modularity.

\subsection{Modular Forms for Fricke Groups}
Let $p$ be either $1$ or a prime number.
The Fricke group of level $p$ is defined by
\[
\Fricke{p}:=\Gz{p}\cup\Gz{p}\begin{psmallmatrix}0&-\frac{1}{\sqrt{p}}\\\sqrt{p}&0\end{psmallmatrix},
\]
where $\Gz{p}:=\{\begin{psmallmatrix}a&b\\c&d\end{psmallmatrix}\in\SLZ\mid c\equiv0\pmod{p}\}$ is the Hecke congruence subgroup of level $p$.
Then the Fricke group $\Fricke{p}$ is a Fuchsian group of the first kind and commensurable with $\SLZ$.

We recall the definition of weakly holomorphic modular forms for the Fricke groups. 
\begin{definition}
    A weakly holomorphic modular form of weight $k\in\Z$ for $\Fricke{p}$ is a holomorphic function $f$ on $\uH$ satisfying the following conditions:
\begin{enumerate}
\item For any $\begin{psmallmatrix}a&b\\c&d\end{psmallmatrix}\in\Fricke{p}$, $f$ satisfies
\[
f\left(\frac{a\tau+b}{c\tau+d}\right)=(c\tau+d)^kf(\tau).
\]
\item There exists $n_0\in\Z$ such that the Fourier series expansion of $f$
\[
f(\tau)=\sum_{n\in\Z}a_f(n)q^n\qquad(a_f(n)\in\C)
\]
satisfies $a_f(n)=0$ for all integers $n<n_0$. 
\end{enumerate}
We denote the space of such forms by $\wM{k}{\Fricke{p}}$.
\end{definition}

\subsection{Fundamental Domain}
Let $\F_p$ denote the standard fundamental domain of $\Fricke{p}$ and let $\A_p$ denote the lower boundary of $\F_p$. 

For $p\in\{1,2,3\}$, the standard fundamental domain of $\Fricke{p}$ is given by
\[
\begin{split}
     \F_p:=&\left\{\tau\in\uH~\middle|~\abs{\tau}\geq\frac{1}{\sqrt{p}},-\frac{1}{2}\leq\Re(\tau)\leq0\right\}\\
     &\quad\cup\left\{\tau\in\uH~\middle|~\abs{\tau}>\frac{1}{\sqrt{p}},0<\Re(\tau)<\frac{1}{2}\right\},
\end{split}
\]
and the lower boundary of $\F_p$ is given by
\[
   \A_p:=\left\{\tau\in\F_p~\middle|~\abs{\tau}=\frac{1}{\sqrt{p}},-\frac{1}{2}\leq\Re(\tau)\leq0\right\}.
\]
For $p\in\{5,7\}$, the standard fundamental domain of $\Fricke{p}$ is given by
\[
\begin{split}
\F_p:=&\left\{\tau\in\uH~\middle|~\abs{\tau}\geq\frac{1}{\sqrt{p}},\abs{\tau+\frac{1}{2}}\geq\frac{1}{2\sqrt{p}},-\frac{1}{2}\leq\Re(\tau)\leq0\right\}\\
&\cup\left\{\tau\in\uH~\middle|~\abs{\tau}>\frac{1}{\sqrt{p}},\abs{\tau-\frac{1}{2}}>\frac{1}{2\sqrt{p}},0<\Re(\tau)<\frac{1}{2}\right\},
\end{split}
\]
and the lower boundary of $\F_p$ is given by
\[
\A_p:=\A_{p,1}\cup\A_{p,2},
\]
where 
\[
\begin{split}
\A_{p,1}&:=\left\{\tau\in\F_p~\middle|~\abs{\tau}=\frac{1}{\sqrt{p}},-\frac{1}{2}\leq\Re(\tau)\leq0\right\},\\
\A_{p,2}&:=\left\{\tau\in\F_p~\middle|~\abs{\tau+\frac{1}{2}}=\frac{1}{2\sqrt{p}},-\frac{1}{2}\leq\Re(\tau)\leq0\right\}.
\end{split}
\]

\subsection{Serre Derivatives and Main Theorem}
Let $k$ be an even integer. For a holomorphic function $f$ on $\mathbb{H}$, the weight $k$ Serre derivative of $f$ is given by
    \[
    \ser_{k}(f)=\ser_{k,p}(f):=\left(D-\frac{p+1}{24}kE_{2,p}\right)f,
    \]
where
\[
E_{2,p}(\tau):=\frac{1}{p+1}\left(pE_2(p\tau)+E_2(\tau)\right)
\]
is the quasi-modular Eisenstein series of weight $2$ for $\Fricke{p}$.
The Serre derivative preserves modularity, as stated in the following proposition.
\begin{proposition}[see Zagier~\cite{Zag08}]\label{prop:ser}
For any weakly holomorphic modular form $f\in\wM{k}{\Fricke{p}}$, we have $\ser_k(f)\in\wM{k+2}{\Fricke{p}}$.
\end{proposition}
The following is our main theorem.
\begin{theorem}\label{thm:main}
Let $p\in\{1,2,3,5,7\}$, let $k$ be an even integer, and let $f\in\wM{k}{\Fricke{p}}$ be a nonzero weakly holomorphic modular form with real Fourier coefficients.
Assume that $f$ has at least one zero in $\uH$ and that all the zeros of $f$ in the standard fundamental domain $\F_p$ lie on the lower boundary $\A_p$.
Then all the zeros of $\ser_{k}(f)$ in $\F_p$ also lie on $\A_p$. 
\end{theorem}
For each integer $n\geq1$, let $\ser_k^{(n)} $ denote $\ser_{k+2(n-1)}\circ\ser_{k+2(n-2)}\circ\cdots\circ\ser_k$.
As a consequence of \Cref{thm:main}, we obtain the following corollary.
\begin{corollary}\label{maincor}
Under the same assumptions as in \Cref{thm:main},  for every integer $n\geq1$, all the zeros of $\ser_k^{(n)}(f)$ in $\F_p$ lie on $\A_p$.
\end{corollary}

This paper is organized as follows. In \Cref{sec:basic}, we recall the valence formula and the Riemann--Hurwitz formula for the Fricke groups.
In \Cref{sec:proof}, we prove \Cref{thm:main} for $p\in\{1,2,3\}$ by analyzing the sign changes of a real-valued and real-analytic function associated with $\ser_k(f)$.
This proof method is inspired by several works on the zeros of quasi-modular forms~\cites{GO22,vIR24,Sug26}.
In \Cref{sec:proof2}, we prove \Cref{thm:main} for $p\in\{5,7\}$.
The cases $p=5,7$ require additional care, since the lower boundary $\A_p$ has two arcs meeting at an elliptic point.
We handle this by carefully analyzing the sign changes of an auxiliary real-valued function near the elliptic point.
In \Cref{sec:proof3}, we prove \Cref{maincor}.
Finally, in \Cref{sec:alt}, we give an alternative proof of our main theorem in the case $p=1$, due to T. Nakaya. The alternative proof proceeds by considering a polynomial associated with a modular form and deriving a contradiction.

\section{Valence Formula and Riemann--Hurwitz Formula}\label{sec:basic}
From this section onward, unless otherwise noted, $p$ denotes either $1$ or a prime number, and $k$ denotes an even integer. We recall the valence formula and the Riemann--Hurwitz formula for the Fricke group $\Fricke{p}$. 

\subsection{Valence Formula}
For a nonzero $f\in\wM{k}{\Fricke{p}}$, $\ord_{\tau}(f)$ denotes the order of $f$ at $\tau\in\uH$. We also define the order of $f$ at the cusp $i\infty$ by
\[
\ord_{i\infty}(f):=\min\{n\in\Z\mid a_f(n)\neq0\}.
\]
The following theorem is known as \textit{the valence formula}. 
\begin{theorem}[Choi--Kim~\cite{CK14}]\label{thm:vf}
For any nonzero $f\in\wM{k}{\Fricke{p}}$, we have
\[
\sum_{\tau\in\F_p}\frac{\ord_{\tau}(f)}{e_{\tau}}+\ord_{i\infty}(f)=\frac{p+1}{24}k,
\]
where $e_{\tau}=e_{\tau,p}$ denotes half of the order of the stabilizer of $\Fricke{p}$ at $\tau\in\F_p$.
\end{theorem}
\begin{remark}\label{rem:cusp}
By the valence formula, a nonzero $f\in \wM{k}{\Fricke{p}}$ has no zeros in $\uH$ if and only if $f$ satisfies $\ord_{i\infty}(f)=\frac{(p+1)}{24}k$. Thus, under the same assumptions as in \Cref{thm:main}, we have $\ord_{i\infty}(f)\neq\frac{p+1}{24}k$. In particular, the Fourier coefficient of $q^{\ord_{i\infty}(f)}$ in $\ser_k(f)$ is nonzero. Hence $\ser_k(f)$ is nonzero and $\ord_{i\infty}(\ser_k(f))=\ord_{i\infty}(f)$.
\end{remark}
\begin{example}[see \cites{MNS07}]
    For $p\in\{1,2,3\}$, the elliptic points of $\Fricke{p}$ in $\F_p$ are $\frac{i}{\sqrt{p}}$ and $\rho_p:=\frac{e^{i\alpha_p}}{\sqrt{p}}$ where $\alpha_p=\frac{p+7}{12}\pi$ . Then the valence formula for $\Fricke{p}$ is given by
\[
\ord_{i\infty}(f)+\frac{1}{2}\ord_{\frac{i}{\sqrt{p}}}(f)+\frac{5-p}{12}\ord_{\rho_p}(f)+\sum_{\substack{\tau\in\F_p\setminus\{ \frac{i}{\sqrt{p}},\rho_p\}}}\ord_{\tau}(f)=\frac{p+1}{24}k.
\]
\end{example}
\begin{example}[see \cites{Shi07}]
The elliptic points of $\Fricke{5}$ in $\F_5$  are $\frac{i}{\sqrt{5}}$, $\rho_{5,1}:=\frac{e^{i\alpha_{5}}}{\sqrt{5}}=-\frac{1}{2}+\frac{e^{i\beta_5}}{2\sqrt{5}}=-\frac{2}{5}+\frac{1}{5}i$ and $\rho_{5,2}:=-\frac{1}{2}+\frac{i}{2\sqrt{5}}$, where $\alpha_5,\beta_5\in(0,\pi)$. Then $\alpha_5-\beta_5=\frac{\pi}{2}$. Moreover, the valence formula for $\Fricke{5}$ is given by
\[
\ord_{i\infty}(f)+\frac{1}{2}\ord_{\frac{i}{\sqrt{5}}}(f)+\frac{1}{2}\ord_{\rho_{5,1}}(f)+\frac{1}{2}\ord_{\rho_{5,2}}(f)+\sum_{\substack{\tau\in\F_5\setminus\{\frac{i}{\sqrt{5}},\rho_{5,1},\rho_{5,2}\}}}\ord_{\tau}(f)=\frac{k}{4}.
\]
\end{example}
\begin{example}[see \cites{Shi07}]
The elliptic points of $\Fricke{7}$ in $\F_7$  are $\frac{i}{\sqrt{7}}$, $\rho_{7,1}:=\frac{e^{i\alpha_{7}}}{\sqrt{7}}=-\frac{1}{2}+\frac{e^{i\beta_7}}{2\sqrt{7}}=-\frac{5}{14}+\frac{\sqrt{3}}{14}i$ and $\rho_{7,2}:=-\frac{1}{2}+\frac{i}{2\sqrt{7}}$, where $\alpha_7,\beta_7\in(0,\pi)$. Then $\alpha_7-\beta_7=\frac{2\pi}{3}$. Moreover, the valence formula for $\Fricke{7}$ is given by
\[
\ord_{i\infty}(f)+\frac{1}{2}\ord_{\frac{i}{\sqrt{7}}}(f)+\frac{1}{3}\ord_{\rho_{7,1}}(f)+\frac{1}{2}\ord_{\rho_{7,2}}(f)+\sum_{\substack{\tau\in\F_7\setminus\{\frac{i}{\sqrt{7}},\rho_{7,1},\rho_{7,2}\}}}\ord_{\tau}(f)=\frac{k}{3}.
\]
\end{example}
We will use the following standard consequence of the local behavior of weakly holomorphic modular forms at elliptic points.
\begin{proposition}\label{prop:ell}
Let $f$ be a nonzero weakly holomorphic modular form of weight $k$ for $\Fricke{p}$.
For any elliptic point $\tau_0\in\F_p$, we have $\ord_{\tau_0}(f)\geq e_{\tau_0}-j$, where $j$ is the unique positive integer satisfying $1\leq j\leq e_{\tau_0}$ and $k\equiv 2j\pmod{2e_{\tau_0}}$.
\end{proposition}

\subsection{Riemann--Hurwitz Formula}
Let $g_p$ denote the genus of the compact Riemann surface
\[
X_0^*(p):=\Fricke{p}\backslash(\uH\cup\mathbb{P}^1(\Q)).
\]
We shall use \textit{the Riemann--Hurwitz formula}, which gives the total contribution of elliptic points in terms of the genus of $X_0^*(p)$~(see Theorem 4.3.1 in \cite{Kat92}). 
\begin{theorem}[The Riemann--Hurwitz Formula for $X_0^*(p)$]\label{thm:RHF}
The following formula holds:
\[
\sum_{\substack{\tau\in\F_p\\\tau\text{:elliptic}}}\left(1-\frac{1}{e_{\tau}}\right)=\frac{p+1}{12}-2g_p+1.
\]
\end{theorem}
\begin{remark}[see Choi--Kim~\cite{CK14}]
By the Riemann--Hurwitz formula, the modular curve $X_0^*(p)$ has genus $0$ if and only if $p=1, 2, 3, 5, 7, 11, 13, 17, 19, 23, 29, 31, 41, 47, 59, 71.$
\end{remark}

\section{Proof of \Cref{thm:main} for $p=1,2,3$}\label{sec:proof}
In this section, we prove \Cref{thm:main} for $p=1,2,3$ by observing the sign changes of a real-valued and real-analytic function associated with the Serre derivative of a modular form.

\subsection{Key Lemma}
First, we record an elementary lemma that will be used in the proof of \Cref{thm:main}; it follows from the local factorization of real-valued and real-analytic functions at their zeros.

\begin{lemma}\label{lem:key}
Let $g$ be a nonzero real-valued and real-analytic function on the open interval $(a,b)$ with finitely many zeros, and let $\theta_1<\theta_2<\cdots<\theta_{n}$ be all the distinct zeros of $g$. Define
\[
G(\theta):=\prod_{j=1}^{n}(\theta-\theta_j)^{m_j-1},
\]
where $m_j:=\ord_{\theta_j}(g)$. Then, for any $j\in\{1,\ldots,n-1\}$, we have
\[
\sgn\frac{g'}{G}(\theta_j)=-\sgn\frac{g'}{G}(\theta_{j+1}),
\]
where $g':=\frac{dg}{d\theta}$. Here $\frac{g'}{G}$ is understood after removing the removable singularities at the zeros of $g$. 
\end{lemma}
\begin{proof}
Put
\[
P(\theta):=\prod_{j=1}^{n}(\theta-\theta_j)^{m_j}.
\]
Since $g$ is real-valued and real-analytic and its zeros are precisely $\theta_1,\ldots,\theta_n$, the function
\[
H(\theta):=\frac{g(\theta)}{P(\theta)}
\]
extends to a real-valued and real-analytic function on $(a,b)$ which has no zeros, and hence has a constant sign on $(a,b)$. By differentiating $g(\theta)=P(\theta)H(\theta)$ with respect to $\theta$, we have
\[
\frac{g'(\theta)}{G(\theta)}=\frac{P'(\theta)H(\theta)+P(\theta)H'(\theta)}{G(\theta)}.
\]
By the definition of $G(\theta)$, it follows that
\[
\begin{split}
    \frac{g'}{G}(\theta_j)&=\frac{P'H+PH'}{G}(\theta_j)=\frac{P'H}{G}(\theta_j)=m_jH(\theta_j)\prod_{l\neq j}(\theta_j-\theta_l).
\end{split}
\]
Therefore, we have
\[
\sgn{\frac{g'}{G}(\theta_j)}=(-1)^{n-j}\sgn H(\theta_j),
\]
and similarly
\[
\sgn{\frac{g'}{G}(\theta_{j+1})}=(-1)^{n-(j+1)}\sgn H(\theta_{j+1}).
\]
Since $H$ has a constant sign on $(a,b)$, we obtain
\[
\sgn{\frac{g'}{G}(\theta_j)}=-\sgn{\frac{g'}{G}(\theta_{j+1})}.
\]
This proves \Cref{lem:key}.
\end{proof}
\begin{remark}
When applying \Cref{lem:key} to zeros lying on a closed subinterval, we choose a slightly larger open interval containing these zeros and no additional zeros.
\end{remark}

\subsection{Proof of \Cref{thm:main} for $p=1,2,3$}
For a holomorphic function $f$ on $\uH$, we define the function $F_{k}(f;\theta)=F_{k,p}(f;\theta)$ on $(0,\pi)$ by
\[
F_k(f;\theta)=F_{k,p}(f;\theta):=e^{ik\theta/2}f\left(\frac{e^{i\theta}}{\sqrt{p}}\right).
\]
By the holomorphy of $f$, the function $F_k(f;\theta)$ is real-analytic on $(0,\pi)$.
The following lemma shows that, if $f\in\wM{k}{\Fricke{p}}$ has real Fourier coefficients, then $F_k(f;\theta)$ satisfies the assumptions of \Cref{lem:key}.
\begin{lemma}\label{lem:real1}
If $f\in\wM{k}{\Fricke{p}}$ has real Fourier coefficients, then $F_k(f;\theta)$ is real-valued on $(0,\pi)$. 
\end{lemma}
\begin{proof}
Since $f$ has real Fourier coefficients, we have
\[
\overline{f(\tau)}=\sum_{n}\overline{a_f(n)e^{2\pi in\tau}}=\sum_{n}a_f(n)e^{-2\pi in\overline{\tau}}=f(-\overline{\tau}).
\]
Combining this with
\[
f\left(-\frac{1}{p\tau}\right)=(\sqrt{p}\tau)^kf(\tau),
\]
where $\tau=e^{i\theta}/\sqrt{p}$, we obtain
\[
\overline{F_k(f;\theta)}=\overline{e^{ik\theta/2}f\left(\frac{e^{i\theta}}{\sqrt{p}}\right)}=e^{-ik\theta/2}f\left(-\frac{1}{p\cdot \frac{e^{i\theta}}{\sqrt{p}}}\right)=e^{-ik\theta/2}(e^{i\theta})^{k}f\left(\frac{e^{i\theta}}{\sqrt{p}}\right)=F_k(f;\theta).
\]
Thus, $F_k(f;\theta)$ is real-valued on $(0,\pi)$.
This proves \Cref{lem:real1}.
\end{proof}

With this notation in place, we now prove \Cref{thm:main} for $p=1,2,3$.
\begin{proof}[Proof of \Cref{thm:main} for $p=1,2,3$]
By \Cref{rem:cusp}, $\ser_k(f)$ is nonzero and
\[
\ord_{i\infty}(\ser_k(f))=\ord_{i\infty}(f).
\]
Since $\ser_k(f)$ is holomorphic on $\uH$, we have 
\[
\ord_{\tau}(\ser_k(f))\geq0
\]
for every $\tau\in\F_p$. By the valence formula for $\ser_k(f)$, we obtain
\[
\sum_{\tau\in\A_p}\frac{\ord_{\tau}(\ser_k(f))}{e_{\tau}}\leq\sum_{\tau\in\F_p}\frac{\ord_{\tau}(\ser_k(f))}{e_{\tau}}= \frac{p+1}{24}(k+2)-\ord_{i\infty}(\ser_k(f)).
\]
Therefore, to prove \Cref{thm:main}, it is enough to show that
\[
\sum_{\tau\in\A_p}\frac{\ord_{\tau}(\ser_k(f))}{e_{\tau}}\geq\frac{p+1}{24}(k+2)-\ord_{i\infty}(f).
\]

Let 
\[
\theta_1<\theta_2<\cdots<\theta_N
\]
be all the distinct zeros of $F_k(f;\theta)$ on $[\frac{\pi}{2},\alpha_p]$, and let $\tau_j$ be the point corresponding to $\theta_j$, namely
\[
\tau_j:=\frac{e^{i\theta_j}}{\sqrt{p}}.
\]
We also define
\[
G(\theta):=\prod_{j=1}^{N}(\theta-\theta_j)^{\ord_{\tau_j}(f)-1}. 
\]

By differentiating $F_k(f;\theta)$ with respect to $\theta$, we have
\[
F_k'(f;\theta)=\frac{ik}{2}F_k(f;\theta)-\frac{2\pi}{\sqrt{p}}F_{k+2}(D(f);\theta).
\]
By the definition of the Serre derivative $\ser_k(f)$, we have
\[
F_{k+2}(D(f);\theta)=F_{k+2}(\ser_k(f);\theta)+\frac{p+1}{24}kF_{2}(E_{2,p};\theta)F_{k}(f;\theta).
\]
It follows that
\begin{equation}\label{eq:eq}
F_{k+2}(\ser_k(f);\theta)+\frac{p+1}{24}kF_{2}(E_{2,p};\theta)F_{k}(f;\theta)=\frac{\sqrt{p}}{2\pi}\left(-F_k'(f;\theta)+\frac{ik}{2}F_k(f;\theta)\right).
\end{equation}
By \Cref{lem:key}, for each $1\leq j\leq N-1$, we have
\begin{equation}\label{eq:eq2}
\sgn\frac{F_k'(f;\cdot)}{G}(\theta_j)=-\sgn\frac{F_k'(f;\cdot)}{G}(\theta_{j+1}).
\end{equation}
Combining \Cref{eq:eq,eq:eq2}, we have
\[
\sgn\frac{F_{k+2}(\ser_k(f);\cdot)}{G}(\theta_j)=-\sgn\frac{F_{k+2}(\ser_k(f);\cdot)}{G}(\theta_{j+1}).
\]
By the intermediate value theorem, there exists $c_j\in (\theta_j,\theta_{j+1})$ such that $F_{k+2}(\ser_k(f);c_j)=0$. 

By the definition of the Serre derivative, we have $\ord_{\tau_j}(\ser_k(f))=\ord_{\tau_j}(f)-1$. Furthermore, for every elliptic point $\tau\in\F_p$ with $f(\tau)\neq0$, we have $\ord_\tau(\ser_k(f))\geq e_{\tau}-1$. Indeed, if $f(\tau)\neq0$ at an elliptic point $\tau$, then \Cref{prop:ell} implies $k\equiv0\pmod{2e_{\tau}}$. Applying \Cref{prop:ell} to $\ser_k(f)$, which has weight $k+2$, gives $\ord_\tau(\ser_k(f))\geq e_{\tau}-1$.

Thus, the weighted number of zeros of $\ser_k(f)$ in $\A_p$ is bounded below by
\begin{equation}\label{eq:lower1}
\begin{split}
\sum_{\tau\in\A_p}\frac{\ord_{\tau}(\ser_k(f))}{e_{\tau}}&\geq \sum_{j=1}^{N}\frac{\ord_{\tau_j}(f)-1}{e_{\tau_j}}+N-1+\sum_{\substack{\tau\in\F_p\\\tau\text{:elliptic}}}\frac{e_{\tau}-1}{e_{\tau}}\delta_{f(\tau)\neq0}\\
&=\sum_{j=1}^{N}\frac{\ord_{\tau_j}(f)}{e_{\tau_j}}+\sum_{j=1}^{N}\left(1-\frac{1}{e_{\tau_j}}\right)-1+\sum_{\substack{\tau\in\F_p\\\tau\text{:elliptic}}}\frac{e_{\tau}-1}{e_{\tau}}\delta_{f(\tau)\neq0}\\
&=\sum_{j=1}^{N}\frac{\ord_{\tau_j}(f)}{e_{\tau_j}}+\sum_{\substack{\tau\in\F_p\\\tau\text{:elliptic}}}\left(1-\frac{1}{e_{\tau}}\right)-1\\
&=\frac{p+1}{24}k-\ord_{i\infty}(f)+\frac{p+1}{12}\\
&=\frac{p+1}{24}(k+2)-\ord_{i\infty}(\ser_k(f))
\end{split}
\end{equation}
where $\delta_{f(\tau)\neq0}$ is equal to $1$ if $f(\tau)\neq0$ and equal to $0$ otherwise.  
The last two equalities in \Cref{eq:lower1} follow from the valence formula, the Riemann--Hurwitz formula, the fact that $g_p=0$ for $p=1,2,3$, and \Cref{rem:cusp}.
This proves \Cref{thm:main} for $p=1,2,3$.
\end{proof}

\section{Proof of \Cref{thm:main} for $p=5,7$}\label{sec:proof2}
In this section, we complete the proof of \Cref{thm:main} for \(p=5,7\).
In these cases, the lower boundary $\A_p$ is the union of two arcs $\A_{p,1}$ and $\A_{p,2}$, which meet at the elliptic point corresponding to $\theta=\alpha_p$.
Therefore, the argument in \Cref{sec:proof} needs to be modified slightly.

For a holomorphic function $f$ on $\uH$, we define
\[
F_{k}^{p,1}(f;\theta):=e^{ik\theta/2}f\left(\frac{e^{i\theta}}{\sqrt{p}}\right)\qquad (0<\theta<\pi),
\]
and
\[
F_{k}^{p,2}(f;\theta):=e^{ik(\theta-\alpha_p+\beta_p)/2}f\left(-\frac{1}{2}+\frac{e^{i(\theta-\alpha_p+\beta_p)}}{2\sqrt{p}}\right)\qquad
(\alpha_p-\beta_p<\theta<\pi+\alpha_p-\beta_p).
\]
We also define 
\[
F_{k}(f;\theta)=F_{k,p}(f;\theta):=\begin{cases}
{F}_{k}^{p,1}(f;\theta), &\text{if }\theta\in(0,\alpha_p],\\
F_{k}^{p,2}(f;\theta),&\text{if }\theta\in(\alpha_p,\pi+\alpha_p-\beta_p).
\end{cases}
\]
By the holomorphy of $f$, the functions $F_k^{p,1}(f;\theta)$ and $F_k^{p,2}(f;\theta)$ are real-analytic on $(0,\pi)$ and on $(\alpha_p-\beta_p,\pi+\alpha_p-\beta_p)$, respectively. 
The following lemma shows that, if $f\in\wM{k}{\Fricke{p}}$ has real Fourier coefficients, then $F_k^{p,1}(f;\theta)$ and $F_{k}^{p,2}(f;\theta)$ satisfy the assumptions of \Cref{lem:key}.
\begin{lemma}\label{lem:real2}
Let $f\in\wM{k}{\Fricke{p}}$ have real Fourier coefficients. Then the following hold:
\begin{enumerate}
    \item $F_k^{p,1}(f;\theta)$ is real-valued on $(0,\pi)$.
    \item $F_{k}^{p,2}(f;\theta)$ is real-valued on $(\alpha_p-\beta_p,\pi+\alpha_p-\beta_p)$.
\end{enumerate}
\end{lemma}
\begin{proof}
The first assertion can be proved in exactly the same way as \Cref{lem:real1}. Thus, we only consider the second assertion. It is enough to show that the function
\[
g(\theta):=e^{ik\theta/2}f\left(-\frac{1}{2}+\frac{e^{i\theta}}{2\sqrt{p}}\right)\qquad(0<\theta<\pi)
\]
is real-valued on $(0,\pi)$.

Set
\[
\tau=-\frac{1}{2}+\frac{e^{i\theta}}{2\sqrt{p}}.
\]
Then a straightforward calculation gives
\[
-\overline{\tau}=\frac{\sqrt{p}\tau+\frac{p-1}{2\sqrt{p}}}{2\sqrt{p}\tau+\sqrt{p}}.
\]
Since $f\in\wM{k}{\Fricke{p}}$ has real Fourier coefficients and $\begin{psmallmatrix}
    \sqrt{p}&\frac{p-1}{2\sqrt{p}}\\2\sqrt{p}&\sqrt{p}
\end{psmallmatrix}\in\Fricke{p}$, we have
\[
\overline{f(\tau)}=f(-\overline{\tau})=f\left(\frac{\sqrt{p}\tau+\frac{p-1}{2\sqrt{p}}}{2\sqrt{p}\tau+\sqrt{p}}\right)=(2\sqrt{p}\tau+\sqrt{p})^{k}f(\tau).
\]
It follows that
\[
\begin{split}
    \overline{g(\theta)}&=\overline{e^{ik\theta/2}f(\tau)}=e^{-ik\theta/2}(2\sqrt{p}\tau+\sqrt{p})^{k}f(\tau)=e^{ik\theta/2}f(\tau)=g(\theta).
\end{split}
\]
Thus, $g$ is real-valued on $(0,\pi)$. This proves the second assertion. 
\end{proof}

With this notation in place, we now prove \Cref{thm:main} for $p=5,7$.
\begin{proof}[Proof of \Cref{thm:main} for $p=5,7$]
Let
\[
\theta_1<\cdots<\theta_{M}
\]
be all the distinct zeros of $F_k(f;\theta)$ on $[\frac{\pi}{2},\alpha_p]$, and let
\[
\theta_{M+1}<\cdots<\theta_{M+N}
\]
be all the distinct zeros of $F_k(f;\theta)$ on $(\alpha_p,\alpha_p-\beta_p+\frac{\pi}{2}]$. For each $j$, let $\tau_j\in\A_p$ be the point corresponding to $\theta_j$, namely
\[
\tau_j:=\begin{cases}
    \frac{e^{i\theta_j}}{\sqrt{p}}&\text{if }j=1,\ldots,M,\\
    -\frac{1}{2}+\frac{e^{i(\theta_j-\alpha_p+\beta_p)}}{2\sqrt{p}}&\text{if }j=M+1,\ldots,M+N.
\end{cases}
\]
We also define
\[
G(\theta):=\prod_{j=1}^{M+N}(\theta-\theta_j)^{\ord_{\tau_j}(f)-1}.
\]

As in \Cref{sec:proof}, \Cref{lem:key} implies that, if $M\geq2$, then for every $1\leq j\leq M-1$, there exists $c_{j}\in(\theta_j,\theta_{j+1})$ such that $F_{k+2}(\ser_k(f);c_j)=F_{k+2}^{p,1}(\ser_k(f);c_j)=0$. Moreover, if $N\geq2$, then for every $1\leq j\leq N-1$, there exists $c_{M+j}\in(\theta_{M+j},\theta_{M+j+1})$ such that $F_{k+2}(\ser_k(f);c_{M+j})=F_{k+2}^{p,2}(\ser_k(f);c_{M+j})=0$. 
For each such $c_j$, let $z_j\in\A_p$ be the corresponding zero of $\ser_k(f)$, namely
\[
z_j:=
\begin{cases}
    \frac{e^{ic_j}}{\sqrt{p}}&\text{if }j=1,\ldots,M-1,\\
    -\frac{1}{2}+\frac{e^{i(c_j-\alpha_p+\beta_p)}}{2\sqrt{p}}&\text{if }j=M+1,\ldots,M+N-1.
\end{cases}
\]

As in \Cref{sec:proof}, \Cref{prop:ell} gives $\ord_\tau(\ser_k(f))\geq e_{\tau}-1$ for every elliptic point $\tau\in\F_p$ with $f(\tau)\neq0$.
By the definition of the Serre derivative, we have $\ord_{\tau_j}(\ser_k(f))=\ord_{\tau_j}(f)-1$.
In addition, by \Cref{rem:cusp}, $\ser_k(f)$ is nonzero and $\ord_{i\infty}(\ser_k(f))=\ord_{i\infty}(f)$.

We divide the proof into two cases:
(1) $MN=0$, and (2) $MN\neq0$. 
\subsection*{Case (1)}
If $MN=0$, the weighted number of zeros of $\ser_k(f)$ in $\A_p$ is bounded below by
\[
\begin{split}
\sum_{\tau\in\A_p}\frac{\ord_{\tau}(\ser_{k}(f))}{e_{\tau}}&\geq\sum_{j=1}^{M+N}\frac{\ord_{\tau_j}(f)-1}{e_{\tau_j}}+M+N-1+\sum_{\substack{\tau\in\F_p\\\tau\text{:elliptic}}}\frac{e_{\tau}-1}{e_{\tau}}\delta_{f(\tau)\neq0}\\
&=\frac{p+1}{24}(k+2)-\ord_{i\infty}(\ser_k(f)).
\end{split}
\]
The last equality follows from the same computation as in \Cref{eq:lower1}.
Together with the valence formula and \Cref{rem:cusp}, this proves \Cref{thm:main} in Case (1).

\subsection*{Case (2)}
 If $MN\neq0$, the weighted number of zeros of $\ser_k(f)$ in $\A_p$ is bounded below by
 \begin{equation}\label{eq:ad}
    \begin{split}
     \sum_{\tau\in\A_p}\frac{\ord_{\tau}(\ser_k(f))}{e_{\tau}}&\geq\sum_{j=1}^{M+N}\frac{\ord_{\tau_j}(f)-1}{e_{\tau_j}}+(M-1)+(N-1)+\sum_{\substack{\tau\in\F_p\\\tau\text{:elliptic}}}\frac{e_{\tau}-1}{e_{\tau}}\delta_{f(\tau)\neq0}\\
&=\frac{p+1}{24}(k+2)-\ord_{i\infty}(\ser_{k}(f))-1.
 \end{split}
 \end{equation}
The last equality follows from the same computation as in \Cref{eq:lower1}.

We now show that any additional zero of $\ser_k(f)$ also lies on $\A_p$. 
Let $Z_{\ser_k(f)}\subset\F_p$ denote the set of zeros of $\ser_k(f)$ in $\F_p$ that are not contained in
\[
L_{\ser_k(f)}:=\{\tau_1,\ldots,\tau_{M+N}\}
\cup
\{z_1,\ldots,z_{M-1},z_{M+1},\ldots,z_{M+N-1}\}
\cup
\left\{\frac{i}{\sqrt p},\rho_{p,1},\rho_{p,2}\right\}.
\]
Since $Z_{\ser_k(f)}$ does not contain elliptic points of $\Fricke{p}$, we have $e_z=1$ for any $z\in Z_{\ser_k(f)}$. Moreover, by the valence formula and \Cref{eq:ad}, we have
\[
\sum_{z\in Z_{\ser_k(f)}}\frac{\ord_{z}(\ser_k(f))}{e_{z}}\leq1.
\]
It follows that $\#Z_{\ser_k(f)}\leq1$.
We shall use the following lemma, whose proof will be given later.
\begin{lemma}\label{lem:57}
With the notation above, we have $Z_{\ser_k(f)}\subset\A_p$.
\end{lemma}
Since $L_{\ser_k(f)}\subset\A_p$, \Cref{lem:57} implies
\[
\{z\in\F_p\mid\ser_k(f)(z)=0\}\subset \A_p.
\]
Hence \Cref{thm:main} follows.
\end{proof}

We now turn to the proof of \Cref{lem:57}.
\begin{proof}[Proof of \Cref{lem:57}]
The assertion is trivial if $Z_{\ser_k(f)}=\emptyset$.
Thus, we may assume that $\#Z_{\ser_k(f)}=1$.
It is enough to show that $F_{k+2}(\ser_k(f);\theta)$ has a zero either in $(\theta_M,\alpha_p)$ or in $(\alpha_p,\theta_{M+1})$.
Indeed, the zero of $\ser_k(f)$ corresponding to such a zero of $F_{k+2}(\ser_k(f);\theta)$ is contained in $\A_p$, but not contained in $L_{\ser_k(f)}$.
We divide the proof into two cases:
(1) $\theta_M=\alpha_p$ and (2) $\theta_M\neq\alpha_p$. 

\subsection*{Case (1)}
If $\theta_M=\alpha_p$, then $F_{k}^{p,2}(f;\theta_M)=0$. By \Cref{lem:key}, it follows that there exists $c_M\in(\theta_M,\theta_{M+1})$ such that $F_{k+2}(\ser_k(f);c_M)=F_{k+2}^{p,2}(\ser_k(f);c_M)=0$. 

\subsection*{Case (2)}
If $\theta_M\neq\alpha_p$, then $f(\rho_{p,1})\neq0$. Hence, by \Cref{prop:ell}, we have 
\[
k\equiv 0\pmod{2e_{\rho_{p,1}}}.
\]
By \Cref{prop:ell}, we have
\[
\ord_{\rho_{p,1}}(\ser_k(f))\ge e_{\rho_{p,1}}-1.
\]
By \Cref{eq:ad}, the valence formula, and the assumption $\#Z_{\ser_k(f)}=1$, this inequality must be an equality. Hence
\[
\ord_{\rho_{p,1}}(\ser_k(f))=e_{\rho_{p,1}}-1.
\]
It follows that
\[
F_{k+2}^{p,1}(\ser_k(f);\alpha_p)=0=F_{k+2}^{p,2}(\ser_k(f);\alpha_p).
\]
Thus $F_{k+2}(\ser_k(f);\theta)$ is continuous at $\theta=\alpha_p$. We consider the subcases $p=5$ and $p=7$ separately.

\subsection*{Subcase (2)--1} 
If $p=5$, we have $e_{\rho_{5,1}}=2$, $k\equiv0\pmod{4}$, and $\alpha_5-\beta_5=\frac{\pi}{2}$. By the definition of $F_{k}^{5,1}(f;\theta)$ and $F_{k}^{5,2}(f;\theta)$, we have
\[
F_{k}^{5,1}(f;\alpha_5)=(-1)^{\frac{k}{4}}F_{k}^{5,2}(f;\alpha_5).
\]
Combining this relation with \Cref{lem:key}, we obtain
\[
\sgn\frac{F_{k+2}(\ser_k(f);\cdot)}{G}(\theta_{M})=(-1)^{\frac{k}{4}+1}\sgn\frac{F_{k+2}(\ser_k(f);\cdot)}{G}(\theta_{M+1}).
\]
By differentiating $F_{k+2}(\ser_k(f);\theta)$ with respect to $\theta$, we have
\[
F_{k+2}'(\ser_k(f);\theta)=\begin{cases}
\begin{aligned}
&\frac{i(k+2)}{2}F_{k+2}^{5,1}(\ser_k(f);\theta)\\&~~~~~~~~~~-\frac{2\pi}{\sqrt{5}}F_{k+4}^{5,1}(D(\ser_k(f));\theta)
\end{aligned}&\theta\in(0,\alpha_5),\\
\begin{aligned}
&\frac{i(k+2)}{2}F_{k+2}^{5,2}(\ser_k(f);\theta)\\&~~~~~~~~~~-\frac{\pi}{\sqrt{5}}F_{k+4}^{5,2}(D(\ser_k(f));\theta)
\end{aligned}&\theta\in(\alpha_5,\pi+\alpha_5-\beta_5).
\end{cases}
\]
Since $\ord_{\rho_{5,1}}\ser_k(f)=1$, the left- and right-hand limits of $F_{k+2}'(\ser_k(f);\theta)$ at $\theta=\alpha_5$ are nonzero. Using the definitions of $F_{k}^{5,1}$ and $F_{k}^{5,2}$, together with $\alpha_5-\beta_5=\frac{\pi}{2}$, we have
\[
\lim_{\theta\to\alpha_5-0}\sgn \frac{F_{k+2}'(\ser_k(f);\theta)}{G(\theta)}=(-1)^{\frac{k}{4}+1}\lim_{\theta\to\alpha_5+0}\sgn\frac{F_{k+2}'(\ser_k(f);\theta)}{G(\theta)}.
\]
It follows that
\[
\lim_{\theta\to\alpha_5-0}\sgn \frac{F_{k+2}(\ser_k(f);\theta)}{G(\theta)}=(-1)^{\frac{k}{4}}\lim_{\theta\to\alpha_5+0}\sgn\frac{F_{k+2}(\ser_k(f);\theta)}{G(\theta)}.
\]
Therefore, one of the following holds:~
\[
\begin{split}
    \lim_{\theta\to\alpha_5-0}\sgn \frac{F_{k+2}(\ser_k(f);\theta)}{G(\theta)}&=-\sgn \frac{F_{k+2}(\ser_k(f);\cdot)}{G}(\theta_M),\\
\lim_{\theta\to\alpha_5+0}\sgn \frac{F_{k+2}(\ser_k(f);\theta)}{G(\theta)}&=-\sgn \frac{F_{k+2}(\ser_k(f);\cdot)}{G}(\theta_{M+1}).
\end{split}
\]
Thus, we obtain an additional zero of $F_{k+2}(\ser_k(f);\theta)$ either in $(\theta_M,\alpha_5)$ or in $(\alpha_5,\theta_{M+1})$.

\subsection*{Subcase (2)--2} 
If $p=7$, we have $e_{\rho_{7,1}}=3$ and $\alpha_7-\beta_7=\frac{2\pi}{3}$. Since $\theta_M\neq\alpha_7$, we have $f(\rho_{7,1})\neq0$. Hence, by \Cref{prop:ell}, the weight $k$ of $f$ satisfies $k\equiv0\pmod{6}$. By the definition of $F_{k}^{7,1}(f;\theta)$ and $F_{k}^{7,2}(f;\theta)$, we have 
\[
\begin{split}
F_{k}^{7,1}(f;\alpha_7)&=F_{k}^{7,2}(f;\alpha_7).
\end{split}
\]
Combining this relation with \Cref{lem:key}, we obtain
\[
\sgn \frac{F_{k+2}(\ser_k(f);\cdot)}{G}(\theta_M)=-\sgn \frac{F_{k+2}(\ser_k(f);\cdot)}{G}(\theta_{M+1}).
\]
Since $\ord_{\rho_{7,1}}\ser_k(f)=2$, the left- and right-hand limits of $F_{k+2}''(\ser_k(f);\theta)$ at $\theta=\alpha_7$ are nonzero.
Using the definitions of $F_{k}^{7,1}$ and $F_{k}^{7,2}$, together with $\alpha_7-\beta_7=\frac{2\pi}{3}$, we have
\[
\lim_{\theta\to\alpha_7-0}\sgn \frac{F_{k+2}''(\ser_k(f);\theta)}{G(\theta)}=\lim_{\theta\to\alpha_7+0}\sgn \frac{F_{k+2}''(\ser_k(f);\theta)}{G(\theta)},
\]
which implies
\[
\lim_{\theta\to\alpha_7-0}\sgn \frac{F_{k+2}(\ser_k(f);\theta)}{G(\theta)}=\lim_{\theta\to\alpha_7+0}\sgn \frac{F_{k+2}(\ser_k(f);\theta)}{G(\theta)}.
\]
As in the case $p=5$, we obtain an additional zero of $F_{k+2}(\ser_k(f);\theta)$ either in $(\theta_M,\alpha_7)$ or in $(\alpha_7,\theta_{M+1})$. 
\end{proof}

\begin{remark}
It would be interesting to extend \Cref{thm:main} to all genus-zero Fricke groups $\Fricke{p}$.
We expect that the same conclusion remains valid in this generality, although the geometry of the lower boundary of the standard fundamental domain becomes more complicated for larger $p$.
\end{remark}

\section{Proof of \Cref{maincor}}\label{sec:proof3}

In this section, we prove \Cref{maincor}.
The key point is that \Cref{thm:main} can be applied repeatedly, because the Serre derivative preserves modularity and real Fourier coefficients.
\begin{proof}[Proof of \Cref{maincor}]
First, we note that, by \Cref{prop:ser} and \Cref{rem:cusp}, $\ser_k(f)$ is a nonzero weakly holomorphic modular form of weight $k+2$ for $\Fricke{p}$ with $\ord_{i\infty}(\ser_k(f))=\ord_{i\infty}(f)$.

By the valence formula, we have
\[
\sum_{\tau\in\F_p}\frac{\ord_{\tau}(\ser_k(f))}{e_{\tau}}=\frac{p+1}{24}(k+2)-\ord_{i\infty}(\ser_k(f))=\frac{p+1}{12}+\left(\frac{p+1}{24}k-\ord_{i\infty}(f)\right)>0.
\]
Indeed, since $f$ has at least one zero in $\uH$, the valence formula implies
\[
\sum_{\tau\in\F_p}\frac{\ord_{\tau}(f)}{e_{\tau}}=\frac{p+1}{24}k-\ord_{i\infty}(f)>0.
\]
Therefore, $\ser_k(f)$ has at least one zero in $\uH$.  By \Cref{thm:main}, all the zeros of $\ser_k(f)$ in $\F_p$ lie on the lower boundary $\A_p$. 
Since the Serre derivative preserves real Fourier coefficients, $\ser_k(f)$ satisfies the assumptions of \Cref{thm:main}.
Thus, we can apply the same argument repeatedly.
\end{proof}

\section{An Alternative Proof of \Cref{thm:main} for $p=1$}\label{sec:alt}

In this section, we give an alternative proof of \Cref{thm:main} for $p=1$, due to T. Nakaya. This proof is independent of the sign-change argument used in the previous sections and is based instead on the polynomial expression of modular forms in terms of the $j$-invariant. This method is elementary, but it requires several case distinctions.

Let
\[
\Delta(\tau):=\frac{1}{1728}(E_4(\tau)^3-E_6(\tau)^2)
\]
be the modular discriminant, and let
\[
j(\tau):=\frac{E_4(\tau)^3}{\Delta(\tau)}\in\wM{0}{\SLZ}
\]
be the classical modular $j$-invariant.
It is known that, for any $\tau\in\F_1$, $\tau\in\A_1$ holds if and only if $j(\tau)\in[0,1728]$.
Moreover, for any $f\in\wM{k}{\SLZ}$, there exists a unique polynomial $P_f\in\C[X]$ such that
\[
f(\tau)=E_{4}(\tau)^\varepsilon E_6(\tau)^\delta\Delta(\tau)^{m}P_f(j(\tau))
\]
where $(\varepsilon,\delta,m)\in\{0,1,2\}\times\{0,1\}\times\Z$ is a triple of  integers satisfying $k=4\varepsilon+6\delta+12m$.

\subsection{An Alternative Proof of \Cref{thm:main} for $p=1$}
We now give a sketch of an alternative proof of \Cref{thm:main} for $p=1$. For a nonzero $f\in\wM{k}{\SLZ}$ with real Fourier coefficients, let $P_f\in\mathbb{R}[X]$ denote the unique polynomial satisfying 
\[
f(\tau)=E_4(\tau)^\varepsilon E_6(\tau)^\delta\Delta(\tau)^mP_f(j(\tau)),
\]
where $(\varepsilon,\delta,m)\in \{0,1,2\}\times\{0,1\}\times\Z$ with $k=4\varepsilon+6\delta+12m$, and let $P_{\ser_k(f)}(X)\in\mathbb{R}[X]$ be the unique polynomial satisfying 
\[
\ser_k(f)(\tau)=E_4(\tau)^{\varepsilon'}E_6(\tau)^{\delta'}\Delta(\tau)^{m'}P_{\ser_k(f)}(j(\tau)),
\]
where $(\varepsilon',\delta',m')\in \{0,1,2\}\times\{0,1\}\times\Z$ with $k+2=4\varepsilon'+6\delta'+12m'$.

Recall that, in $\F_1$, the Eisenstein series $E_4(\tau)$ and $E_6(\tau)$ have their unique zeros at $\rho_1$ and $i$, respectively, and that these zeros are simple.
Moreover, $\Delta(\tau)$ has no zeros in $\uH$. 
Thus, it suffices to prove the following reformulation of \Cref{thm:main} for $p=1$. 

\begin{theorem}\label{thm:main2}
Under the same assumptions as in \Cref{thm:main} with $p=1$, all the roots of $P_{\ser_k(f)}$ lie in the closed interval $[0,1728]$.
\end{theorem}
\begin{proof}[Sketch of Proof]
By Ramanujan's differential equations, we have
\[
\ser_4(E_4)=-\frac{1}{3}E_6,\qquad\ser_6(E_6)=-\frac{1}{2}E_4^2,\qquad\ser_{12}(\Delta)=0.
\]
Using these identities, a straightforward calculation gives
\[
\ser_{k}(f)(\tau)=E_{4}(\tau)^\varepsilon E_6(\tau)^\delta \frac{E_6(\tau)}{E_4(\tau)}\Delta(\tau)^{m}Q_f(j(\tau)),
\]
where
\[
\begin{split}
    Q_f(X)&:=c_{\varepsilon,\delta}(X)P_f(X)-X \frac{d P_{f}}{dX}(X),\\
    c_{\varepsilon,\delta}(X)&:=-\frac{\varepsilon}{3}-\frac{\delta}{2}\frac{X}{X-1728}.
\end{split}
\]

Since the essence of the proof is the same in all cases, we only consider the case $(\varepsilon,\delta)=(1,0)$.
In this case, we have $(\varepsilon',\delta')=(0,1)$ and
\[
P_{\ser_k(f)}(X)=Q_f(X)=-\frac{1}{3}P_f(X)-X\frac{dP_f}{dX}(X).
\]
By the definition of $Q_f(X)$, if $0\in S_{Q_f}$, then $0\in S_{P_f}$, where $S_P$ denotes the set of roots of $P(X)\in\C[X]$.
By the assumptions of \Cref{thm:main}, we have $S_{P_f}\subset[0,1728]$, in particular $S_{Q_f}\cap S_{P_f}\subset [0,1728]$.

We now define the rational function $R_f(X)$ by
\[
R_f(X):=-\frac{Q_f(X)}{X P_f(X)}=\frac{1}{3X}+\frac{d\log P_f}{dX}(X)=\frac{1}{3X}+\sum_{z\in S_{P_f}}\frac{\ord_{z}P_f}{X-z}.
\]
Let $z=a+bi\in S_{Q_f}\setminus S_{P_f}$, where $a,b\in\mathbb{R}$.
Then we have $R_f(z)=0$.
Taking imaginary parts of $R_f(z)$, we have
\[
\Im(R_f(z) )=-b\left(\frac{1}{3\abs{z}^2}+\sum_{w\in S_{P_f}}\frac{\ord_{w}P_f}{\abs{z-w}^2}\right)=0,
\]
which implies $b=0$. Thus $z\in\mathbb{R}$. 

If $z<0$, then $z-w<0$ for all $w\in S_{P_f}$, and hence
\[
R_f(z)=\frac{1}{3z}+\sum_{w\in S_{P_f}}\frac{\ord_{w}P_f}{z-w}<0,
\]
which contradicts $R_f(z)=0$. Similarly, if $z>1728$, then $z-w>0$ for all $w\in S_{P_f}$, and hence $R_f(z)>0$, again contradicting $R_f(z)=0$. This shows that $z\in[0,1728]$.
Since $S_{Q_f}\cap S_{P_f}\subset[0,1728]$, it follows that $S_{Q_f}\subset[0,1728]$.

The remaining cases are treated in the same way as the case $(\varepsilon,\delta)=(1,0)$, with only slight changes. Indeed, using
\[
\frac{E_6(\tau)^2}{E_4(\tau)^3}=\frac{j(\tau)-1728}{j(\tau)},
\]
one obtains $P_{\ser_k(f)}$ explicitly by a straightforward calculation. We therefore omit the details.
\end{proof}

\begin{remark}
With a slight modification, this proof can also be applied to the cases $p=2,3,5,7$.
\end{remark}

\section*{Acknowledgements}
The author is grateful to Dr.~Akihiro Goto for suggesting this research topic.
The author also thanks Prof.~Masanobu Kaneko, Dr.~Yuichi Sakai, Dr.~Tomoaki Nakaya, and Dr.~Toshiteru Kinjo for their helpful comments and suggestions on earlier versions of this paper.
The author is particularly indebted to Dr.~Tomoaki Nakaya for providing an alternative proof of \Cref{thm:main} with $p=1$, which is included in \Cref{sec:alt}.


\begin{bibdiv}
\begin{biblist}

\bib{ACN16}{article}{
      author={Aricheta, Victor~Manuel},
      author={Celeste, Richell},
      author={Nemenzo, Fidel},
       title={Hecke operators and zeros of modular forms},
        date={2016},
        ISSN={0126-6705,2180-4206},
     journal={Bull. Malays. Math. Sci. Soc.},
      volume={39},
      number={3},
       pages={1249\ndash 1257},
         url={https://doi.org/10.1007/s40840-016-0333-3},
      review={\MR{3515078}},
}

\bib{AKN97}{article}{
      author={Asai, Tetsuya},
      author={Kaneko, Masanobu},
      author={Ninomiya, Hirohito},
       title={Zeros of certain modular functions and an application},
        date={1997},
        ISSN={0010-258X},
     journal={Comment. Math. Univ. St. Paul.},
      volume={46},
      number={1},
       pages={93\ndash 101},
      review={\MR{1448475}},
}

\bib{BKM08}{article}{
      author={Bannai, Eiichi},
      author={Kojima, Koji},
      author={Miezaki, Tsuyoshi},
       title={On the zeros of {H}ecke-type {F}aber polynomials},
        date={2008},
        ISSN={1340-6116,1883-2032},
     journal={Kyushu J. Math.},
      volume={62},
      number={1},
       pages={15\ndash 61},
         url={https://doi.org/10.2206/kyushujm.62.15},
      review={\MR{2413781}},
}

\bib{CI16}{article}{
      author={Choi, SoYoung},
      author={Im, Bo-Hae},
       title={On the zeros of certain weakly holomorphic modular forms for
  {$\Gamma_0^+(2)$}},
        date={2016},
        ISSN={0022-314X,1096-1658},
     journal={J. Number Theory},
      volume={166},
       pages={298\ndash 323},
         url={https://doi.org/10.1016/j.jnt.2016.02.008},
      review={\MR{3486279}},
}

\bib{CK14}{article}{
      author={Choi, SoYoung},
      author={Kim, Chang~Heon},
       title={Valence formulas for certain arithmetic groups and their
  applications},
        date={2014},
        ISSN={0022-247X,1096-0813},
     journal={J. Math. Anal. Appl.},
      volume={420},
      number={1},
       pages={447\ndash 463},
      review={\MR{3229834}},
}

\bib{DJ08}{article}{
      author={Duke, W.},
      author={Jenkins, Paul},
       title={On the zeros and coefficients of certain weakly holomorphic
  modular forms},
        date={2008},
        ISSN={1558-8599,1558-8602},
     journal={Pure Appl. Math. Q.},
      volume={4},
      number={4},
       pages={1327\ndash 1340},
      review={\MR{2441704}},
}

\bib{Get04}{article}{
      author={Getz, Jayce},
       title={A generalization of a theorem of {R}ankin and {S}winnerton-{D}yer
  on zeros of modular forms},
        date={2004},
        ISSN={0002-9939,1088-6826},
     journal={Proc. Amer. Math. Soc.},
      volume={132},
      number={8},
       pages={2221\ndash 2231},
      review={\MR{2052397}},
}

\bib{Gun06}{article}{
      author={Gun, Sanoli},
       title={On the zeros of certain cusp forms},
        date={2006},
        ISSN={0305-0041,1469-8064},
     journal={Math. Proc. Cambridge Philos. Soc.},
      volume={141},
      number={2},
       pages={191\ndash 195},
         url={https://doi.org/10.1017/S0305004106009522},
      review={\MR{2265867}},
}

\bib{GO22}{article}{
      author={Gun, Sanoli},
      author={Oesterl\'e, Joseph},
       title={Critical points of {E}isenstein series},
        date={2022},
        ISSN={0025-5793,2041-7942},
     journal={Mathematika},
      volume={68},
      number={1},
       pages={259\ndash 298},
      review={\MR{4405978}},
}

\bib{Han23}{article}{
      author={Hanamoto, Seiichi},
       title={Zeros of certain weakly holomorphic modular forms},
        date={2023},
        ISSN={1340-6116,1883-2032},
     journal={Kyushu J. Math.},
      volume={77},
      number={2},
       pages={255\ndash 270},
         url={https://doi.org/10.2206/kyushujm.77.255},
      review={\MR{4665180}},
}

\bib{HK21}{article}{
      author={Hanamoto, Seiichi},
      author={Kuga, Seiji},
       title={Zeros of certain weakly holomorphic modular forms for the
  {F}ricke group {$\Gamma _0^+(3)$}},
        date={2021},
        ISSN={0065-1036,1730-6264},
     journal={Acta Arith.},
      volume={197},
      number={1},
       pages={37\ndash 54},
         url={https://doi.org/10.4064/aa190509-7-2},
      review={\MR{4185914}},
}

\bib{KZ98}{incollection}{
      author={Kaneko, M.},
      author={Zagier, D.},
       title={Supersingular {$j$}-invariants, hypergeometric series, and
  {A}tkin's orthogonal polynomials},
        date={1998},
   booktitle={Computational perspectives on number theory ({C}hicago, {IL},
  1995)},
      series={AMS/IP Stud. Adv. Math.},
      volume={7},
   publisher={Amer. Math. Soc., Providence, RI},
       pages={97\ndash 126},
         url={https://doi.org/10.1090/amsip/007/05},
      review={\MR{1486833}},
}

\bib{Kat92}{book}{
      author={Katok, Svetlana},
       title={Fuchsian groups},
      series={Chicago Lectures in Mathematics},
   publisher={University of Chicago Press, Chicago, IL},
        date={1992},
        ISBN={0-226-42582-7; 0-226-42583-5},
      review={\MR{1177168}},
}

\bib{Kla19}{article}{
      author={Klangwang, Jetjaroen},
       title={Zeros of certain combinations of {E}isenstein series of weight
  {$2k$}, {$3k$}, and {$k+l$}},
        date={2019},
        ISSN={0022-314X,1096-1658},
     journal={J. Number Theory},
      volume={198},
       pages={124\ndash 138},
         url={https://doi.org/10.1016/j.jnt.2018.10.005},
      review={\MR{3912932}},
}

\bib{Kug21}{article}{
      author={Kuga, Seiji},
       title={On the zeros of certain weakly holomorphic modular forms for
  {$\Gamma_0^+(5)$} and {$\Gamma_0^+(7)$}},
        date={2021},
        ISSN={0065-1036,1730-6264},
     journal={Acta Arith.},
      volume={201},
      number={3},
       pages={219\ndash 239},
         url={https://doi.org/10.4064/aa200318-3-4},
      review={\MR{4361573}},
}

\bib{Kug22}{article}{
      author={Kuga, Seiji},
       title={On the distribution of the zeros of {E}isenstein series for
  {$\Gamma^*_0(5)$} and {$\Gamma^*_0(7)$}},
        date={2022},
        ISSN={1382-4090,1572-9303},
     journal={Ramanujan J.},
      volume={58},
      number={4},
       pages={995\ndash 1010},
         url={https://doi.org/10.1007/s11139-022-00557-5},
      review={\MR{4451508}},
}

\bib{MNS07}{article}{
      author={Miezaki, Tsuyoshi},
      author={Nozaki, Hiroshi},
      author={Shigezumi, Junichi},
       title={On the zeros of {E}isenstein series for {$\Gamma^*_0(2)$} and
  {$\Gamma^*_0(3)$}},
        date={2007},
        ISSN={0025-5645,1881-1167},
     journal={J. Math. Soc. Japan},
      volume={59},
      number={3},
       pages={693\ndash 706},
      review={\MR{2344823}},
}

\bib{RSD70}{article}{
      author={Rankin, F. K.~C.},
      author={Swinnerton-Dyer, H. P.~F.},
       title={On the zeros of {E}isenstein series},
        date={1970},
        ISSN={0024-6093,1469-2120},
     journal={Bull. London Math. Soc.},
      volume={2},
       pages={169\ndash 170},
      review={\MR{260674}},
}

\bib{Ran69}{article}{
      author={Rankin, R.~A.},
       title={The zeros of {E}isenstein series},
        date={1969},
        ISSN={0304-9965},
     journal={Publ. Ramanujan Inst.},
      volume={1},
       pages={137\ndash 144},
      review={\MR{269598}},
}

\bib{Ran82}{article}{
      author={Rankin, R.~A.},
       title={The zeros of certain {P}oincar\'e{} series},
        date={1982},
        ISSN={0010-437X,1570-5846},
     journal={Compositio Math.},
      volume={46},
      number={3},
       pages={255\ndash 272},
      review={\MR{664646}},
}

\bib{Rav25}{article}{
      author={Raveh, Roei},
       title={On the zeros of the {M}iller basis of cusp forms},
        date={2025},
        ISSN={2522-0160,2363-9555},
     journal={Res. Number Theory},
      volume={11},
      number={4},
       pages={Paper No. 96, 32},
         url={https://doi.org/10.1007/s40993-025-00679-x},
      review={\MR{4978581}},
}

\bib{RVY17}{article}{
      author={Reitzes, Sarah},
      author={Vulakh, Polina},
      author={Young, Matthew~P.},
       title={Zeros of certain combinations of {E}isenstein series},
        date={2017},
        ISSN={0025-5793,2041-7942},
     journal={Mathematika},
      volume={63},
      number={2},
       pages={666\ndash 695},
         url={https://doi.org/10.1112/S0025579317000110},
      review={\MR{3706602}},
}

\bib{Rud05}{article}{
      author={Rudnick, Ze\'ev},
       title={On the asymptotic distribution of zeros of modular forms},
        date={2005},
        ISSN={1073-7928,1687-0247},
     journal={Int. Math. Res. Not.},
      number={34},
       pages={2059\ndash 2074},
         url={https://doi.org/10.1155/IMRN.2005.2059},
      review={\MR{2181743}},
}

\bib{Shi07}{article}{
      author={Shigezumi, Junichi},
       title={On the zeros of the {E}isenstein series for {$\Gamma^*_0(5)$} and
  {$\Gamma^*_0(7)$}},
        date={2007},
        ISSN={1340-6116,1883-2032},
     journal={Kyushu J. Math.},
      volume={61},
      number={2},
       pages={527\ndash 549},
      review={\MR{2362897}},
}

\bib{Shi10}{article}{
      author={Shigezumi, Junichi},
       title={On the zeros of certain {P}oincar\'e{} series for
  {$\Gamma_0^*(2)$} and {$\Gamma_0^*(3)$}},
        date={2010},
        ISSN={0030-6126},
     journal={Osaka J. Math.},
      volume={47},
      number={2},
       pages={487\ndash 505},
         url={http://projecteuclid.org/euclid.ojm/1277298914},
      review={\MR{2722370}},
}

\bib{Sug26}{article}{
      author={Sugibayashi, Naoki},
       title={Critical points of the {E}isenstein series for the {F}ricke group
  of level 2},
        date={2026},
        ISSN={1793-0421,1793-7310},
     journal={Int. J. Number Theory},
      volume={22},
      number={3},
       pages={489\ndash 518},
         url={https://doi.org/10.1142/S1793042126500296},
      review={\MR{5026411}},
}

\bib{vIR24}{article}{
      author={van Ittersum, Jan-Willem},
      author={Ringeling, Berend},
       title={Critical points of modular forms},
        date={2024},
        ISSN={1793-0421,1793-7310},
     journal={Int. J. Number Theory},
      volume={20},
      number={10},
       pages={2695\ndash 2728},
      review={\MR{4833417}},
}

\bib{Woh63}{article}{
      author={Wohlfahrt, Klaus},
       title={{\"U}ber die {N}ullstellen einiger {E}isensteinreihen},
        date={1963/64},
        ISSN={0025-584X,1522-2616},
     journal={Math. Nachr.},
      volume={26},
       pages={381\ndash 383},
         url={https://doi.org/10.1002/mana.19630260606},
      review={\MR{167470}},
}

\bib{XZ21}{article}{
      author={Xue, Hui},
      author={Zhu, Daozhou},
       title={Location of the zeros of certain cuspforms},
        date={2021},
        ISSN={1793-0421,1793-7310},
     journal={Int. J. Number Theory},
      volume={17},
      number={8},
       pages={1899\ndash 1903},
         url={https://doi.org/10.1142/S1793042121500688},
      review={\MR{4310213}},
}

\bib{XZ25}{article}{
      author={Xue, Hui},
      author={Zhu, Daozhou},
       title={Zeros of cuspidal projections of products of {E}isenstein
  series},
        date={2025},
        ISSN={1382-4090,1572-9303},
     journal={Ramanujan J.},
      volume={68},
      number={4},
       pages={Paper No. 107, 49},
         url={https://doi.org/10.1007/s11139-025-01250-z},
      review={\MR{4989999}},
}

\bib{Zag08}{incollection}{
      author={Zagier, Don},
       title={Elliptic modular forms and their applications},
        date={2008},
   booktitle={The 1-2-3 of modular forms},
      series={Universitext},
   publisher={Springer, Berlin},
       pages={1\ndash 103},
         url={https://doi.org/10.1007/978-3-540-74119-0_1},
      review={\MR{2409678}},
}

\end{biblist}
\end{bibdiv}
\end{document}